\documentclass[reqno,10pt,a4paper]{amsart}
\usepackage[margin=2.5cm]{geometry}
\usepackage{amssymb}
\usepackage{amsthm}
\usepackage{enumitem}
\usepackage{lmodern}
\usepackage[usenames,dvipsnames]{color}
\usepackage{hyperref}
\hypersetup{
 colorlinks,
 linkcolor={red!50!black},
 citecolor={blue!50!black},
 urlcolor={blue!80!black}
}
\usepackage{tikz,tikz-3dplot}

\AtBeginDocument{%
\def\MR#1{}
}
\theoremstyle{plain}
\newtheorem{thm}{Theorem}[section]
\newtheorem{prop}[thm]{Proposition}
\newtheorem{cor}[thm]{Corollary}
\newtheorem{lem}[thm]{Lemma}

\newtheorem{quest}[thm]{Question}
\theoremstyle{definition}

\newtheorem{ex}[thm]{Example}

\newcommand{\Pyr}{{\mathrm{Pyr}}}

\DeclareMathOperator{\conv}{conv}

\DeclareMathOperator{\ehr}{ehr}

\newcommand{\R}{{\mathbb{R}}}
\newcommand{\Z}{{\mathbb{Z}}}

\newcommand{\rleft}{\mathopen{}\mathclose\bgroup\left}
\newcommand{\rright}{\aftergroup\egroup\right}
\newcommand{\vect}[1]{{\boldsymbol{\mathbf{#1}}}}

\newcommand{\orig}{{\vect{0}}}

\DeclareMathOperator{\Vol}{Vol}

\def\eb{\mathbf{e}}
\def\coloneqq{\mathrel{\mathop:}=}
\title{Universal inequalities in Ehrhart Theory}

\author[G.\,Balletti]{Gabriele Balletti}
\address[G.\,Balletti]{Department of Mathematics\\Stockholm University\\SE-$106$\ $91$\
  Stockholm\\Sweden}
\curraddr{}
\email{balletti@math.su.se}
\thanks{}

\author[A.\,Higashitani]{Akihiro Higashitani}
\address[A.\,Higashitani]{Department of Mathematics\\Kyoto Sangyo University\\ 603-8555, Kyoto\\Japan}
\curraddr{}
\email{ahigashi@cc.kyoto-su.ac.jp}
\thanks{}


\begin{document}

\begin{abstract}
In this paper, we show the existence of \emph{universal inequalities} for the $h^*$-vector of a lattice polytope $P$, that is, we show that there are relations among the coefficients of the $h^*$-polynomial which are independent of both the dimension and the degree of $P$. 
More precisely, we prove that the coefficients $h_1^*$ and $h_2^*$ of the $h^*$-vector $(h_0^*,h_1^*,\ldots,h_d^*)$ of a lattice polytope of any degree satisfy Scott's inequality if $h_3^*=0$. 
\end{abstract}

\maketitle{}

\section{Introduction}\label{sec:introduction}

\subsection{Background and main result}\label{ssec:background}
Let $N \cong \Z^d$ be a lattice of rank $d$ and let $N_\R \coloneqq N \otimes_\Z \R \cong \R^d$. 
We say that a convex polytope $P \subset N_\R$ is a \emph{lattice polytope} if its vertices are all elements of $N$.  
Two lattice polytopes $P,Q\subset N_\R$ are said to be \emph{unimodularly equivalent} if there exists an affine lattice automorphism $\varphi\in \mathrm{GL}_d(\Z)\ltimes\Z^d$ of $N$ such that $\varphi_\R(P)=Q$. 
In what follows, unless stated otherwise, we consider lattice polytopes as being defined up to unimodular equivalence. 

Given a lattice polytope $P$ of dimension $d$, one can associate an enumerative function $k \mapsto |kP \cap N|$, counting the number of lattice points in the $k$-th dilation $kP$ of $P$, where $k$ is a positive integer. 
Ehrhart~\cite{Ehr62} proved that this function is interpolated by a polynomial, i.e. that there exists a polynomial $\ehr_P$, called the \emph{Ehrhart polynomial} of $P$, such that $\ehr_P(k)=|kP \cap N|$ for any $k \geq 1$. 
Moreover, its generating function is known to be the rational function
\[\sum_{k \geq 0} \ehr_P(k) t^k =\frac{\sum_{i \geq 0}h_i^*t^i}{(1-t)^{d+1}},\]
where $h_i^*=0$ for any $i \geq d+1$. We call the sequence of integers $(h_0^*,h_1^*,\ldots,h_d^*)$ appearing in the numerator of the generating function the \emph{$h^*$-vector} (or \emph{$\delta$-vector}) of $P$, 
and the polynomial $h_P^*(t)=h^*_0 + h^*_1 t + \cdots + h^*_s t^s$ the \emph{$h^*$-polynomial} (or \emph{$\delta$-polynomial}) of $P$, where $s$ is the degree of this polynomial. We define $s$ to be the \emph{degree} of $P$. 
In the following, we will sometimes use the notation $h_i^*=h^*_i(P)$ when we want to specify the polytope $P$. 
The constant term $h^*_0$ of the $h^*$-polynomial is always $1$. Moreover, the coefficients are all nonnegative integers (see \cite{Sta80}). 
The $h^*$-polynomial (the $h^*$-vector) is an important and meaningful invariant for $P$, and despite (in fixed dimension) it is equivalent to the Ehrhart polynomial of $P$, it is often preferred to $\ehr_P$ as its coefficients have a well-understood combinatorial interpretation. 
For more details, see Section \ref{sec:universal_inequalities} (Proposition~\ref{prop:properties}).

The following question is one of the most important unsolved problems in Ehrhart Theory. 
\begin{quest}
\label{question}
Can one characterize the polynomials with nonnegative integer coefficients that are the $h^*$-polynomial of some lattice polytope?
\end{quest}

This question has a trivial answer in dimension one, while an answer in dimension two has been given by Scott \cite{Sco76}. 

\begin{thm}[Scott \cite{Sco76} (1976)]
\label{thm:Scott}
A polynomial $h^*(t) = 1 + h^*_1 t + h^*_2 t^2 \in \Z_{\geq 0}[t]$ is the $h^*$-polynomial of a lattice polytope of dimension two if and only if it satisfies one of the following conditions: 
\begin{enumerate}
\item \label{scott:1} $h^*_2 = 0$;
\item \label{scott:2} $h_2^* \leq h^*_1 \leq 3h^*_2 + 3$;
\item \label{scott:3} $h^*_1 = 7$ and $h^*_2=1$.
\end{enumerate}
Moreover, the case $\eqref{scott:3}$ is satisfied only by the $h^*$-polynomial of the triangle $\conv(\{\orig,3\eb_1,3\eb_2\})$, where $\eb_1$ and $\eb_2$ are a basis for the ambient lattice and $\orig$ is its origin. 
\end{thm}

All the cases for $d \geq 3$ of Question~\ref{question} are widely open.

In \cite{Tre10}, Treutlein generalizes the necessary conditions of Theorem \ref{thm:Scott} to lattice polytopes of degree at most two as follows. 
\begin{thm}[{\cite[Theorem~2]{Tre10}}]
\label{thm:generalizedScott}
Let $P$ be a lattice polytope of degree at most two. Then its $h^*$-polynomial $h^*_P(t) = 1 + h^*_1 t + h^*_2 t^2$  satisfies one of the following conditions: 
\begin{enumerate}
\item \label{genscott:1} $h^*_2 = 0$;
\item \label{genscott:2} $h^*_1 \leq 3h^*_2 + 3$;
\item \label{genscott:3} $h^*_1 = 7$ and $h^*_2=1$. 
\end{enumerate}
Moreover, the case $\eqref{genscott:3}$ is satisfied only by the $h^*$-polynomial of lattice simplices obtained via multiple lattice pyramid constructions (see $\eqref{eq:lattice_pyramid}$) 
over the triangle $\conv(\{\orig,3\eb_1,3\eb_2\})$, where $\eb_1$ and $\eb_2$ are part of a basis for the ambient lattice and $\orig$ is its origin.
\end{thm}
Henk--Tagami \cite[Proposition~1.10]{HT09} proved that those are sufficient conditions, i.e. that any polynomial of degree two in $\Z_{\geq 0}[t]$ satisfying one of the conditions \eqref{genscott:1}--\eqref{genscott:3} of Theorem~\ref{thm:generalizedScott} is the $h^*$-polynomial of some lattice polytope of degree two.
Note that the inequality $h_2^* \leq h_1^*$ of Theorem~\ref{thm:Scott}, coming from $\eqref{Ehr_ineq}$, does not appear in Theorem~\ref{thm:generalizedScott}, as the dimension of the polytope may be greater than two.

We call the inequalities $\eqref{genscott:1}$--$\eqref{genscott:3}$ in Theorem \ref{thm:generalizedScott} \emph{Scott's inequality}. 
Our main result is the following further generalization of Theorem~\ref{thm:generalizedScott} to polynomials $h^*(t) = 1 + h^*_1 t + h^*_2 t^2 + \cdots \in \Z_{\geq 0}[t]$ of any degree satisfying $h^*_3=0$. 
\begin{thm}[Main Theorem]\label{thm:main}
Let $P$ be a lattice polytope whose $h^*$-polynomial $h_P^*(t) = 1 + h^*_1 t + h^*_2 t^2 + \cdots \in \Z_{\geq 0}[t]$ satisfies $h^*_3=0$. 
Then $h_P^*(t)$ satisfies Scott's inequality, i.e. it satisfies one of the following conditions: 
\begin{enumerate}
\item \label{main:1} $h^*_2 = 0$;
\item \label{main:2} $h^*_1 \leq 3h^*_2 + 3$;
\item \label{main:3} $h^*_1 = 7$ and $h^*_2=1$.
\end{enumerate}
Moreover, the case $\eqref{main:3}$ is satisfied only by the $h^*$-polynomial of lattice simplices obtained via multiple lattice pyramid constructions over the triangle $\conv(\{\orig,3\eb_1,3\eb_2\})$ and then considered as lattice polytopes with respect to a refined lattice. 
\end{thm}

The latter part of the statement of Theorem \ref{thm:main} will be clarified and proven in Proposition~\ref{prop:last}.

Theorem~\ref{thm:main} gives the first relation among the coefficients of the $h^*$-polynomial of a lattice polytope $P$ 
which is valid independently of both the dimension and the degree of $P$. We call this new kind of inequality \emph{universal}. The existence of this kind of inequalities has been conjectured by Benjamin Nill (private communication), who also suggested the term ``universal''.

With the following example, we notice that the condition $h^*_3 = 0$ in Theorem~\ref{thm:main} is necessary.
\begin{ex}
The $5$-dimensional lattice polytope 
\[ \conv(\{\orig,\eb_1,\eb_2,\eb_1+\eb_2+2\eb_3,\eb_4,\eb_4+9\eb_5\}) \subset \R^5, \]
where $\eb_1,\ldots,\eb_5$ are a basis for $\Z^5$, has $1+8t+t^2+8t^3$ as its $h^*$-polynomial. In particular it does not satisfy Scott's inequality. 
\end{ex}

\subsection{Organization of the paper}
The paper is organized as follows. 
Section~\ref{sec:universal_inequalities} is devoted to giving some backgrounds on Ehrhart Theory and known inequalities for $h^*$-vectors of lattice polytopes. The proof of Theorem~\ref{thm:main} will be given in the remaining sections. 
In Section~\ref{sec:spanning}, we prove the main statement of the main theorem for all cases excluding a special one. Such special case needs a technical proof given in Section~\ref{sec:law}. Finally, in Section~\ref{sec:special}, the last part of the statement of Theorem~\ref{thm:main} will be proven.


\section*{Acknowledgments}
The authors would like to thank the PhD advisor of the first author, Benjamin Nill, for posing the question of the existence of universal inequalities, for posing Question~\ref{question2}, and, together with Johannes Hofscheier, for many inspiring discussions. The project started during a visit of the first author at Otto-von-Guericke Universit\"{a}t, Magdeburg. The actual collaboration started while the authors were participants in the ``Einstein Workshop on Lattice Polytopes'' in Berlin. Moreover, the authors would like to thank Takayuki Hibi and Akiyoshi Tsuchiya in Osaka University for their hospitality. 
The first author is partially supported by the Vetenskapsr{\aa}det grant~NT:2014-3991 and the second author is partially supported by JSPS Grant-in-Aid for Young Scientists (B) $\sharp$26800015. 


\section{Inequalities and universal inequalities in Ehrhart Theory}\label{sec:universal_inequalities}

\subsection{Known inequalities in Ehrhart Theory}\label{ss:known_inequalities}
The following proposition summarizes some of the well-known interpretations of the $h^*$-vectors of a lattice polytope. All the following can be deduced from Ehrhart's original approach \cite{Ehr62}.

\begin{prop}
\label{prop:properties}
Let $P$ be a $d$-dimensional lattice polytope of degree $s$ with respect to the lattice $N$. Then the $h^*$-vector $(h^*_0,\ldots,h^*_d)$ of $P$ satisfies 
\[
h^*_0=1, \;\; h^*_1=|P \cap N|-d-1, \;\; h^*_s=|(d+1-s)P^\circ \cap N|, \;\; h^*_d=|P^\circ \cap N|, \;\;\text{and} \;\; \sum_{i=0}^d h^*_i=\Vol(P), 
\]
where by $P^\circ$ we denote the relative interior of $P$ and by $\Vol(P)$ we denote the \emph{normalized volume} of $P$ which equals $d!$ times the Euclidean volume of $P$.
In particular, from the descriptions of $h_1^*$ and $h^*_d$, one gets 
\begin{equation}\label{Ehr_ineq}
h^*_d \leq h^*_1.
\end{equation}
\end{prop}

In the last decades several other relations among the $h^*$-vectors have been proven. Let $P$ be a $d$-dimensional lattice polytope of degree $s$ with its $h^*$-vector $h^*(P)=(h^*_0,\ldots,h^*_d)$. In \cite{Hib94} Hibi proved that
\[\label{eq:_hibi}
h^*_{d-1} + \cdots + h^*_{d-i} \leq h^*_2 + \dots + h^*_{i+1} \quad \mbox{for }i=1,\ldots, d-1;
\]
while Stanley in~\cite{Sta91} proved that
\[\label{eq:_stanley}
h^*_{0} + \dots + h^*_{i} \leq h^*_{s} + \cdots + h^*_{s-i} \quad \mbox{for } i=0,1,\ldots, s.
\]
Another result from Hibi~\cite{Hib94}, which is valid in the case $P$ has interior points (which is equivalent to $h^*_d >0$), is the following:
\[\label{ineq:hibi}
h^*_1 \leq h^*_i \;\;\mbox{for }i=1,2,\ldots,d-1.
\]
More recently Stapledon~\cite{Sta09,Sta16} showed the existence of infinite new classes of inequalities and improved existing ones.

All the mentioned and previously known families of inequalities for the $h^*$-vectors of lattice polytopes have different forms when specialized to different dimensions or degrees. 
Namely, they are not universal.

\subsection{Realizing $h^*$-polynomial}
The most natural way to approach Question~\ref{question} is to fix a polynomial $h^*(t) \in \Z_{\geq 0}[t]$ and try to check whether it is realizable as the $h^*$-polynomial of a lattice polytope. 
As a first step one can try to fix a polynomial $h(t)= 1 + h^*_1 t + \cdots + h^*_k t^k \in \Z_{\geq 0}[t]$ for some $k > 0$ and check whether it is possible to complete it as the $h^*$-polynomial of a lattice polytope 
by adding parts of degree strictly larger than $k$. It is then natural to ask the following question, which firstly has been posed by Benjamin Nill (private communication).
\begin{quest}\label{question2}
Let $h(t)= 1 + h^*_1 t + \cdots + h^*_k t^k \in \Z_{\geq 0}[t]$ for some $k>0$. Is there a polynomial $H(t) \in \Z_{\geq 0}[t]$ such that $h(t)+t^{k+1}H(t)$ is the $h^*$-polynomial of a lattice polytope?
\end{quest}

None of the known inequalities for the $h^*$-vectors forbids such a possibility. Furthermore, in Proposition~\ref{prop:twocoeff}, we answer positively to Question~\ref{question2} for $k \leq 2$. 

A good tool for realizing $h^*$-vectors as just described is the following construction. 
\begin{lem}[{\cite[Lemma 1.3]{HT09}}]
\label{lem:join}
Let $P \subset \R^m$ and $Q \subset \R^n$ be two lattice polytopes. Then the \emph{join} 
\[P \star Q \coloneqq \conv\{ ({\bf x},\orig_n,0),(\orig_m,{\bf y},1) : {\bf x} \in P, \; {\bf y} \in Q \} \subset \R^{n+m+1},\]
where $\orig_i$ denotes the origin in $\R^i$, has the $h^*$-polynomial $h^*_{P \star Q}(t)=h^*_{P}(t)h^*_{Q}(t)$.
\end{lem}

Moreover we are going to need a special class of lattice simplices having binomial $h^*$-polynomial of arbitrary degree. Such family was described by Batyrev--Hofscheier \cite[Theorem~2.5]{BH10}.

\begin{lem}
\label{lem:uncle_johannes_lemma}
For each choice of positive integers $s \geq 1$ and $b \geq 1$, there exists a $(2s-1)$-dimensional lattice simplex $\Delta_{s,b}$ having the $h^*$-polynomial $h^*_{\Delta_{s,b}}(t)=1 + b t^s$.
\end{lem}

\begin{prop}
\label{prop:twocoeff}
Let $h(t)= 1 + a t + b t^2 \in \Z_{\geq 0}[t]$. Then there exists a polynomial $H(t) \in \Z_{\geq 0}[t]$ such that $h(t)+t^3H(t)$ is the $h^*$-polynomial of a lattice polytope. 
In particular, Question~\ref{question2} is true for $k \leq 2$. 
\end{prop}
\begin{proof}
We can choose $H(t) \coloneqq ab$. Indeed, by Lemma~\ref{lem:uncle_johannes_lemma} there exists a three-dimensional lattice simplex $\Delta_{2,b}$ having the $h^*$-polynomial $h^*_{\Delta_{2,b}}(t)=1 + b t^2$. Let $L_a$ be the lattice segment ($1$-dimensional lattice polytope) $L_a \coloneqq \conv(\{0,a+1\}) \subset \R$. 
Then the join $\Delta_{2,b} \star L_a$ has the $h^*$-polynomial $h^*_{\Delta_{2,b} \star L_a}(t)=(1 + b t^2)(1 + a t)=1 + a t + b t^2 + ab t^3$ by Lemma \ref{lem:join}, as required. 
\end{proof}

On the other hand Theorem~\ref{thm:main} answers negatively to all the other cases.

\begin{cor}
Question~\ref{question2} is false for $k \geq 3$.
\end{cor}

Exploiting the construction used in Proposition~\ref{prop:twocoeff}, one can easily create infinite families of polytopes of arbitrarily large degree satisfying Scott's inequality of Theorem~\ref{thm:main}.

\begin{prop}
For each choice of $s \geq 6$ and nonnegative integers $h^*_1,h^*_2$ satisfying the conditions \emph{\eqref{main:1}--\eqref{main:3}} of Theorem~\ref{thm:main}, 
there exists a lattice polytope of degree $s$ for which Theorem~\ref{thm:main} agrees.
\end{prop}
\begin{proof}
By \cite[Proposition~1.10]{HT09}, there exists a polytope $Q$ of degree at most two having the $h^*$-polynomial $h^*_Q(t) = 1 + h^*_1 t + h^*_2 t^2$. By Lemma~\ref{lem:join} and Lemma~\ref{lem:uncle_johannes_lemma}, and for any choice of $k \geq 1$, the join $\Delta_{s-2,k} \star Q$ has the $h^*$-polynomial $h^*_{\Delta_{s-2,k} \star Q}(t)=(1 + k t^{s-2})(1 + h^*_1 t + h^*_2  t^2)=1 + h^*_1 t + h^*_2 t^2 + k t^{s-2} + kh^*_1 t^{s-1} + k h^*_2 t^s$.
\end{proof}

One might be tempted to suspect that the only interesting cases for which Theorem~\ref{thm:main} applies, are these polytopes artificially built via joins, as in the previous proposition. This is not the case: in Example~\ref{ex:non_join} we build a lattice polytope of degree five satisfying Scott's inequality that cannot be obtained via a join construction.


\section{The non-Lawrence prism case}\label{sec:spanning}

In this section we recall the key notion of our proof from \cite{HKN16} and give a proof of Theorem~\ref{thm:main} for the majority of the cases, namely for all the polytopes whose spanning polytope (see definition below) is not a Lawrence prism (see $\eqref{eq:lawrence}$). The Lawrence prism case needs technical proofs and will be treated in Section~\ref{sec:law}.

\subsection{Spanning polytopes}
Let $P \subset \R^d$ be a $d$-dimensional lattice polytope with respect to the lattice $N$ of rank $d$. 
Let us denote by $\widetilde{N}$ the affine sublattice of $N$ generated by the points of $P \cap N$, i.e. $\widetilde{N}$ consists of all the integral affine combinations of $P \cap N$. We define the \emph{spanning polytope $\widetilde{P}$ associated to $P$} as the lattice polytope given by the vertices of $P$ with respect to the lattice $\widetilde{N}$ and we say that $P$ is \emph{spanning} if $\widetilde{N}=N$.

We denote by $h^*_{\widetilde{P}}(t) = 1 + \widetilde{h_1^*}t + \cdots + \widetilde{h_{\tilde{s}}^*}t^{\tilde{s}}$ the $h^*$-polynomial of $\widetilde{P}$. 
Then we have the following inequalities 
\begin{equation}\label{eq:h_spanning}
\widetilde{h_1^*} = h^*_1 \;\; \text{ and }\;\; \widetilde{h_i^*} \le h_i^*  \;\text{ for }\; i \ge 2, 
\end{equation}
and in particular $\tilde{s} \leq s$. (See \cite[Section 3.2]{HKN16}.) 

We use the following recent result by Hofscheier--Katth\"{a}n--Nill.
\begin{thm}[{\cite[Theorem~1.3]{HKN16}}]\label{thm:HKN}
The $h^*$-polynomial $h^*_{\tilde{P}}(t) = 1 + \widetilde{h_1^*}t + \cdots + \widetilde{h_{\tilde{s}}^*} t^{\tilde{s}}$ of any spanning polytope $\widetilde{P}$ satisfies
\[\widetilde{h_i^*} \geq 1 \quad \mbox{for all} \quad i=0, \ldots , \tilde{s}. \]
\end{thm}

In particular the spanning polytope $\widetilde{P}$ of a lattice polytope $P$ having $h^*_3=0$ has degree at most two.

\subsection{Characterization of lattice polytopes with degree one}
We recall the work by Batyrev--Nill \cite{BN07}. Given a $d$-dimensional lattice polytope $Q \subset \R^d$ with respect to $\Z^d$, 
we define the \emph{lattice pyramid} $\Pyr(Q)$ as the $(d+1)$-dimensional polytope
\begin{equation}\label{eq:lattice_pyramid}
\Pyr(Q) \coloneqq \conv(Q \times \{0\} \cup \{(0,\ldots,0,1)\}) \subset \R^{d+1}. 
\end{equation}
The lattice pyramid construction preserves the $h^*$-polynomial, i.e. $h^*_{Q}(t)=h^*_{\Pyr(Q)}(t)$ (\cite[Theorem 2.4]{BR15}). 
We say that a $d$-dimensional lattice polytope $P$ is an \emph{exceptional simplex} if $P$ can be obtained 
via the $(d-2)$-fold iterations of the lattice pyramid construction over the second dilation of a unimodular simplex, that is, 
\[ 
P \cong \Pyr(\cdots(\Pyr(\conv(\{ \orig , 2\eb_1 , 2\eb_2\})))\cdots).
\]
We say that a $d$-dimensional lattice polytope $P$ is a \emph{Lawrence prism} with \emph{heights} $a_0, \ldots , a_{d-1}$ if there exist nonnegative integers $a_0, \ldots , a_{d-1}$ such that 
\[\label{eq:lawrence}
P \cong \conv (\{ \orig , a_0 \eb_d , \eb_1 , \eb_1 + a_1 \eb_d , \ldots , \eb_{d-1}, \eb_{d-1} + a_{d-1} \eb_d  \}).
\]

\begin{thm}[{\cite[Theorem~2.5]{BN07}}]\label{thm:BN}
Let $P$ be a lattice polytope. Then $\deg(P) \leq 1$ if and only if $P$ is an exceptional simplex or a Lawrence prism.
\end{thm}

This classification is a powerful tool. Indeed, as we will see in Proposition \ref{prop:deduce}, the case in which the spanning polytope $\widetilde{P}$ of $P$ has degree one is the only one that cannot be entirely deduced directly from Theorem~\ref{thm:HKN} and Theorem~\ref{thm:generalizedScott}. For this special case the classification result is necessary.

\subsection{A proof for lattice polytopes whose spanning polytope is not a Lawrence prism}
\begin{prop}\label{prop:deduce}
Let $P$ be a lattice polytope whose $h^*$-polynomial $h^*_P(t)=1+h_1^*t+h_2^*t^2+\cdots $ satisfies $h^*_3=0$. 
Assume that the spanning polytope $\widetilde{P}$ of $P$ is not a Lawrence prism. Then the inequalities \eqref{main:1}--\eqref{main:3} in Theorem \ref{thm:main} hold. 
\end{prop}
\begin{proof}
Since $\widetilde{h_3^*} \leq h_3^*$ by \eqref{eq:h_spanning} and $h_3^*=0$, we have $\widetilde{h_3^*}=0$. Then Theorem \ref{thm:HKN} implies that $\widetilde{h_i^*}=0$ for any $i \geq 3$, in particular $\widetilde{P}$ has degree $\tilde{s} \leq 2$.

If $\tilde{s}=2$, by Theorem \ref{thm:generalizedScott}, we see that the $h^*$-vector of $\widetilde{P}$ satisfies Scott's inequality. By $\widetilde{h_1^*}=h_1^*$ and $0 < \widetilde{h_2^*} \leq h_2^*$, we conclude that $P$ also satisfies Scott's inequality. 
If $\tilde{s}=1$, then, by our assumptions and Theorem~\ref{thm:BN}, $\widetilde{P}$ must be an exceptional simplex. In particular we get $h^*_1=\widetilde{h_1^*}=3$, so Scott's inequality is satisfied for any value of $h^*_2$.
Finally, if $\tilde{s}=0$ then $\widetilde{h^*_1}=0$, so Scott's inequality is satisfied for any value of $h^*_2$, as required.
\end{proof}


\section{The Lawrence prism case}\label{sec:law}
In this section, we consider the missing case in which $P$ is a $d$-dimensional lattice polytope with $h^*_3(P)=0$ and whose spanning polytope $\widetilde{P}$ is a Lawrence prism. To prove this case we first show that, if $\Delta'$ and $\Delta''$ are two full-dimensional empty simplices contained in $P$, then $h^*_2(\Delta')=h^*_2(\Delta'')$. 
This follows from Proposition~\ref{prop:emplex}. Above, we call a simplex $S$ \emph{empty} if it has no lattice points other than its vertices, or equivalently, if it satisfies $h^*_1(S)=0$. Later, with an inclusion-exclusion argument, we show (Proposition~\ref{prop:lll}) strong conditions on the coefficients $h^*_1(P)$ and $h^*_2(P)$, which are enough to finish the proof of the main statement of Theorem~\ref{thm:main}.

For the proof, we recall some notation from the following well-known technique for the computation of the $h^*$-vectors of lattice simplices by associating them to finite abelian groups. 
Let $\Delta$ be a lattice simplex with respect to the lattice $N$ and let $v_0,\ldots,v_d \in N$ be the vertices of $\Delta$. 
We define 
\[
\Lambda_\Delta \coloneqq \left\{(r_0,\ldots,r_d) \in [0,1)^{d+1} : \sum_{i=0}^d r_iv_i \in N, \; \sum_{i=0}^d r_i \in \Z_{\geq 0}\right\}.
\]
We see that $\Lambda_\Delta$ is a finite abelian group with its addition 
$\alpha+\beta=(\{\alpha_0+\beta_0\},\ldots,\{\alpha_d+\beta_d\}) \in [0,1)^{d+1}$ for $\alpha,\beta \in \Lambda_\Delta$, 
where $\{r\}=r - \lfloor r \rfloor$ denotes the fractional part of $r \in \R$. 
Note that ${\bf 0}=(0,\ldots,0) \in \Lambda_\Delta$. We define
\begin{equation}
\label{eq:Lambda_i}
\Lambda_\Delta^{(h)} \coloneqq \left\{(r_0,\ldots,r_d) \in \Lambda_\Delta : \sum_{i=0}^{d} r_i=h \right\} \;\; \text{for} \; h=0,1,\ldots,d.
\end{equation}
Note that $\Lambda_\Delta = \bigsqcup_{i=0}^d \Lambda_\Delta^{(i)}$ and 
the $h^*$-vector of $\Delta$ can be computed by
\begin{equation}
\label{eq:h^*}
h_i^*=|\Lambda_\Delta^{(i)}|\;\text{ for each $i$},
\end{equation}
see \cite[Corollary 3.11]{BR15}.

As in Section~\ref{sec:spanning}, we denote by $N$ the ambient lattice of $P$, while $\widetilde{N} \subseteq N$ is the lattice affinely spanned by the points in $P \cap N$. Let $\eb_1,\ldots,\eb_d$ be a basis for a lattice $\widetilde{N}$. Since $\widetilde{P}$ is a Lawrence prism, we assume that $P \subset D$, where $D$ is the unbounded prism
\[
D:= \conv (\{ \orig, \eb_1 , \ldots , \eb_{d-1}\}) \times \R \eb_d \subset \R^d.
\]
Note that the $(d-1)$-dimensional simplex $\conv (\{ \orig, \eb_1 , \ldots , \eb_{d-1}\})$ over which the $D$ is built may not be unimodular simplex if considered with respect to $N$.

\begin{prop}\label{prop:emplex}
With the notation just introduced, let $\Delta'$ and $\Delta''$ be two $d$-dimensional empty simplices contained in $D$, with vertices on $\widetilde{N}$, but considered as lattice polytopes with respect to the refined lattice $N$. Suppose that $h^*_{3}(\Delta') = h^*_{3}(\Delta'')=0$. Then $h^*_{2}(\Delta') = h^*_{2}(\Delta'')$.
\end{prop}
\begin{proof}
Note that any empty simplex $\Delta$ in $D$ having vertices on $\widetilde{N}$ can be easily described. In particular, there exist $0 \leq k \leq d$ and nonnegative integers $c_0,\ldots,c_{d-1}$ 
such that $\Delta$ can be written as 
\[
\Delta=\Delta_k'\coloneqq \conv(\{\eb_0+c_0\eb_d,\eb_1+c_1\eb_d, \ldots, \eb_{d-1}+c_{d-1}\eb_d, \eb_k+(c_k-1)\eb_d\}),
\]
where by $\eb_0$ we denote the origin $\orig$ of $\R^d$. 
In the following, all the polytopes are considered with respect to the refined lattice $N$. We first prove that
\begin{enumerate}[label=(\alph*)]
\item \label{a_lem} $h^*_2(\Delta_k')=h^*_2(\Delta_k)$, where $\Delta_k \coloneqq \conv(\{\eb_0,\ldots,\eb_{d-1},\eb_k+\eb_d\})$;
\end{enumerate}
and we conclude by proving
\begin{enumerate}[label=(\alph*)]
\setcounter{enumi}{1}
\item \label{b_lem} $h^*_2(\Delta_k)=h^*_2(\Delta_0)$
\end{enumerate}
for any $k$. The simplices $\Delta_k'$, $\Delta_k$ and $\Delta_0$ are represented in Figure~\ref{fig:proof_emplex}.

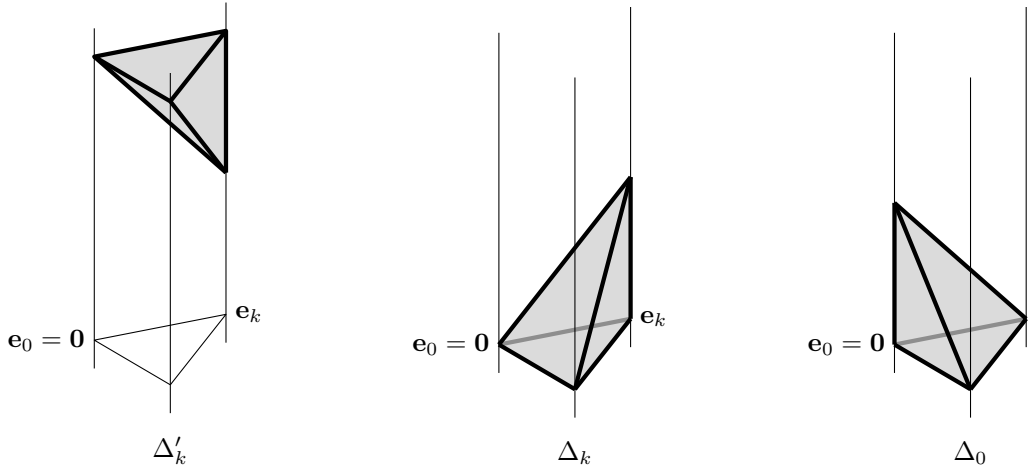
\begin{figure}[h]
\centering
\tdplotsetmaincoords{70}{60}
\begin{tikzpicture}[scale=2,tdplot_main_coords]
	\fill [black!20,opacity=0.7] (0,1,2) -- (0,0,2) -- (0,1,1) -- cycle;
	\draw [ultra thick,black,line join=round] (0,0,2) -- (1,0,2) -- (0,1,2) -- cycle;
	\draw [ultra thick,black,line join=round] (0,0,2) -- (0,1,1);
	\draw [ultra thick,black,line join=round] (1,0,2) -- (0,1,1);
	\draw [ultra thick,black,line join=round] (0,1,2) -- (0,1,1);

	\draw [black,line join=round] (0,0,0) -- (1,0,0) -- (0,1,0) -- cycle;
	\draw [black,line join=round] (0,0,-.2) -- (0,0,2.2);
	\draw [black,line join=round] (1,0,-.2) -- (1,0,2.2);
	\draw [black,line join=round] (0,1,-.2) -- (0,1,2.2);
	\draw[] (0,1,0) node[right]{$\eb_k$};
	\draw[] (0,0,0) node[left]{$\eb_0=\orig$};
	\draw[] (1,0,-.3) node[below]{$\Delta_k'$};
\end{tikzpicture}
\hspace*{1.5cm}
\begin{tikzpicture}[scale=2,tdplot_main_coords]
	\draw [ultra thick,black,line join=round] (0,0,0) -- (0,1,0);
	\draw [black,line join=round] (0,0,0) -- (1,0,0) -- (0,1,0) -- cycle;
	\fill [black!20,opacity=0.7] (0,0,0) -- (1,0,0) -- (0,1,0) -- (0,1,1) -- cycle;

	\draw [ultra thick,black,line join=round] (0,0,0) -- (1,0,0);
	\draw [ultra thick,black,line join=round] (0,1,0) -- (1,0,0);
	\draw [ultra thick,black,line join=round] (0,0,0) -- (0,1,1);
	\draw [ultra thick,black,line join=round] (1,0,0) -- (0,1,1);
	\draw [ultra thick,black,line join=round] (0,1,0) -- (0,1,1);

	\draw [black,line join=round] (0,0,-.2) -- (0,0,2.2);
	\draw [black,line join=round] (1,0,-.2) -- (1,0,2.2);
	\draw [black,line join=round] (0,1,-.2) -- (0,1,2.2);
	\draw[] (0,1,0) node[right]{$\eb_k$};
	\draw[] (0,0,0) node[left]{$\eb_0=\orig$};
	\draw[] (1,0,-.3) node[below]{$\Delta_k$};
\end{tikzpicture}\hspace*{1.5cm}
\begin{tikzpicture}[scale=2,tdplot_main_coords]
	\draw [ultra thick,black,line join=round] (0,0,0) -- (0,1,0);
	\draw [black,line join=round] (0,0,0) -- (1,0,0) -- (0,1,0) -- cycle;
	
	\fill [black!20,opacity=0.7] (0,0,0) -- (1,0,0) -- (0,1,0) -- (0,0,1) -- cycle;
	
	\draw [ultra thick,black,line join=round] (0,0,0) -- (1,0,0);
	\draw [ultra thick,black,line join=round] (0,1,0) -- (1,0,0);
	\draw [ultra thick,black,line join=round] (0,0,0) -- (0,0,1);
	\draw [ultra thick,black,line join=round] (1,0,0) -- (0,0,1);
	\draw [ultra thick,black,line join=round] (0,1,0) -- (0,0,1);

	\draw [black,line join=round] (0,0,-.2) -- (0,0,2.2);
	\draw [black,line join=round] (1,0,-.2) -- (1,0,2.2);
	\draw [black,line join=round] (0,1,-.2) -- (0,1,2.2);
	\draw[] (0,1,0) node[right]{$\phantom{\eb_k}$};
	\draw[] (0,0,0) node[left]{$\eb_0=\orig$};
	\draw[] (1,0,-.3) node[below]{$\Delta_0$};
\end{tikzpicture}
\caption{\label{fig:proof_emplex} The empty simplices $\Delta_k'$, $\Delta_k$ and $\Delta_0$ in the unbounded prism $D$.}
\end{figure}

Let
{\footnotesize
\begin{align*}
&\Lambda_{\Delta_k}=\left\{(r_0,\ldots,r_d) \in [0,1)^{d+1} : \sum_{i=0}^{d-1} r_i\eb_i + r_d(\eb_k+\eb_d) \in N, \; \sum_{i=0}^d r_i \in \Z\right\}, \\
&\Lambda_{\Delta_k'}=\Bigg\{(r_0,\ldots,r_d) \in [0,1)^{d+1} : \sum_{\substack{0 \leq i \leq d-1 \\ i \neq k}} r_i(\eb_i+c_i\eb_d) + 
r_k(\eb_k+(c_k-1)\eb_d)+r_d(\eb_k+c_k\eb_d) \in N, \; \sum_{i=0}^d r_i \in \Z\Bigg\}. 
\end{align*}
}
Let $\Lambda_{\Delta_k}^{(i)}$ and $\Lambda_{\Delta_k'}^{(i)}$ be as in \eqref{eq:Lambda_i} for $i=0,1,\ldots,d$. 

We first show that there exists a bijection $\pi$ between $\Lambda_{\Delta_k}$ and $\Lambda_{\Delta_k'}$ mapping $\Lambda_{\Delta_k}^{(2)}$ to $\Lambda_{\Delta_k'}^{(2)}$. 
Let $(r_0,\ldots,r_d)$ be any element of $\Lambda_{\Delta_k}$, we define $\pi:\Lambda_{\Delta_k} \rightarrow \Lambda_{\Delta_k'}$ by setting
{\footnotesize
\begin{align*}
\pi((r_0,\ldots,r_d)) =& (\pi(r_0),\ldots,\pi(r_d)) \\
\coloneqq & \left(r_0,\ldots,r_{k-1},\left\{\sum_{i=0}^{d-1}r_ic_i+r_d(c_k-1)\right\},r_{k+1},\ldots,r_{d-1},
\left\{ r_k-r_d(c_k-2)-\sum_{i=0}^{d-1}r_ic_i \right\}\right).
\end{align*}
}
We now check that $\pi$ is well-defined. It is straightforward to verify that $\sum_{i=0}^d \pi(r_i)$ becomes an integer when $\sum_{i=0}^d r_i$ is an integer. Let $(r_0,\ldots,r_d) \in \Lambda_{\Delta_k}$. By definition, $\sum_{i=0}^{d-1} r_i\eb_i + r_d(\eb_k+\eb_d)\in N$. Thus,
{\footnotesize
\begin{align*}
&\sum_{\substack{0 \leq i \leq d-1\\ i \neq k}} \pi(r_i)(\eb_i+c_i\eb_d)+\pi(r_k)(\eb_k+(c_k-1)\eb_d)+\pi(r_d)(\eb_k+c_k\eb_d) \\
\equiv & \sum_{\substack{0 \leq i \leq d-1\\ i \neq k}} r_i(\eb_i+c_i\eb_d)+\left(\sum_{i=0}^{d-1}r_ic_i + r_d(c_k-1) \right)(\eb_k+(c_k-1)\eb_d)
+ \left(r_k - r_d(c_k-2) - \sum_{i=0}^{d-1}r_ic_i\right)(\eb_k+c_k\eb_d) \\
= & \sum_{\substack{0 \leq i \leq d-1\\ i \neq k}} r_i\eb_i + r_k\eb_k + r_d(\eb_k+\eb_d) \in N \mod\;\widetilde{N}. 
\end{align*}
}
Similarly, we can also construct the inverse map $\pi^{-1}:\Lambda_{\Delta_k'} \rightarrow \Lambda_{\Delta_k}$. We set
{\footnotesize
\begin{align*}
\pi^{-1}((r_0,\ldots,r_d)) =& (\pi^{-1}(r_0),\ldots,\pi^{-1}(r_d)) \\
\coloneqq & \left(r_0,\ldots,r_{k-1},\left\{ 2r_k+r_d(1-c_k)-\sum_{i=0}^{d-1}r_ic_i\right\},r_{k+1},\ldots,r_{d-1},
\left\{ -r_k+\sum_{i=0}^{d-1}r_ic_i+r_dc_k \right\}\right). 
\end{align*}
}
Also in this case the map is well-defined.
Note that $0 \leq \pi(r_k),\pi(r_d),\pi^{-1}(r'_k),\pi^{-1}(r'_d) < 1$. This means that, for $h \geq 1$,
\[
\pi(\Lambda_{\Delta_k}^{(h)}) \subset \Lambda_{\Delta'_k}^{(h-1)} \cup \Lambda_{\Delta'_k}^{(h)} \cup \Lambda_{\Delta'_k}^{(h+1)} \quad \text{and} \quad \pi^{-1}(\Lambda_{\Delta_k'}^{(h)}) \subset \Lambda_{\Delta_k}^{(h-1)} \cup \Lambda_{\Delta_k}^{(h)} \cup \Lambda_{\Delta_k}^{(h+1)}.
\]
In particular, since by our assumptions $\Lambda_{\Delta_k}^{(1)}=\Lambda_{\Delta_k}^{(3)}=\Lambda_{\Delta_k'}^{(1)}=\Lambda_{\Delta_k'}^{(3)}=\emptyset$, we see that $\pi$ induces a bijection between $\Lambda_{\Delta_k}^{(2)}$ and $\Lambda_{\Delta'_k}^{(2)}$. This proves \ref{a_lem}.

We prove \ref{b_lem} similarly. Let $\psi:\Lambda_{\Delta_0} \rightarrow \Lambda_{\Delta_k}$ be the map defined by 
\[
\psi((r_0,\ldots,r_d)):=\left(\{r_0+r_d\}, r_1,\ldots,r_{k-1},\{r_k-r_d\},r_{k+1},\ldots,r_d \right).
\]
Then $\sum_{i=0}^d \psi(r_i) \in \Z$ if $\sum_{i=0}^d r_i \in \Z$. Let $(r_0,\ldots,r_d) \in \Lambda_{\Delta_0}$. Then $\sum_{i=0}^d r_i\eb_i \in N$. Moreover, 
\begin{align*}
\sum_{i=0}^{d-1} \psi(r_i)\eb_i + \psi(r_d)(\eb_k+\eb_d) 
\equiv \sum_{i=1}^{d-1} r_i\eb_i - r_d\eb_k + r_d(\eb_k+\eb_d) = \sum_{i=1}^d r_i\eb_i \in N \mod\; \widetilde{N}. 
\end{align*}
Hence, $\psi : \Lambda_{\Delta_0} \rightarrow \Lambda_{\Delta_k}$ is well-defined. We can also construct the inverse $\psi^{-1}:\Lambda_{\Delta_k} \rightarrow \Lambda_{\Delta_0}$ by 
\[
\psi^{-1}((r'_0,\ldots,r'_d)):=\left(\left\{r'_0-r'_d\right\}, r'_1,\ldots,r'_{k-1},\{r'_k+r'_d\},r'_{k+1},\ldots,r'_d \right).
\]
Exactly as in case \ref{a_lem}, although $\psi$ might not induce a bijection between $\Lambda_{\Delta_0}^{(h)}$ and $\Lambda_{\Delta_k}^{(h)}$ for $h>3$, a bijection is induced for $h=2$. This proves \ref{b_lem}.
\end{proof}

We can now prove that Theorem~\ref{thm:main} holds also in the Lawrence prism case.
 
\begin{prop}\label{prop:lll}
Let $P$ be a $d$-dimensional lattice polytope whose $h^*$-polynomial is of the form $h^*_P(t)=1+h_1^*t+h_2^*t^2+\cdots $ with $h^*_3=0$. Suppose that its spanning polytope $\widetilde{P}$ is a Lawrence prism. Then $h_2^*$ is divisible by $h^*_1+1$. In particular, the inequalities \emph{\eqref{main:1}--\eqref{main:3}} in Theorem \ref{thm:main} hold. 
\end{prop}
\begin{proof}Since $\widetilde{P}$ is a Lawrence prism, $\widetilde{P}$ is of the form 
\[
\widetilde{P}=\conv(\{\eb_i,\eb_i+a_i\eb_d : 0 \leq i \leq d-1\}) \subset \R^d,
\]
where $a_0,a_1,\ldots,a_{d-1}$ are nonnegative integers and $\eb_0$ denotes the origin of $\R^d$. 
We regard $\widetilde{P}$ as a lattice polytope with respect to the sublattice $\widetilde{N} \subseteq N$, affinely spanned by the points of $P \cap N$. We know that the $h^*$-polynomial $h^*_{\widetilde{P}}$ of $\widetilde{P}$ has degree at most one, i.e.  $h^*_{\widetilde{P}} = 1 + \widetilde{h_1^*}t$. Note that since $\widetilde{h_1^*}=h_1^*$ the degree zero case is trivial, so we can assume $h^*_1 >0$.

Fix nonnegative integers $a_0,a_1,\ldots,a_{d-1}$. 
We define the lattice simplex $\Delta_{i,j}$ for $0 \leq i \leq d-1$, $1 \leq j \leq a_i$ as the lattice polytope
\[
\Delta_{i,j}\coloneqq\conv(\{\eb_0, \ldots,\eb_{i-1},\eb_i+j\eb_d,\eb_i+(j-1)\eb_d,\eb_{i+1}+a_{i+1}\eb_d,\ldots,\eb_{d-1}+a_{d-1}\eb_d\}),
\]
considered with respect to the lattice $N$. 
Then $\Delta_{i,j} \subset P$ for any $i,j$. Hence, by monotonicity \cite[Theorem 3.3]{Sta93}, we have $h_3^*(\Delta_{i,j})=0$. Moreover, $\Delta_{i,j}$ is an empty simplex, equivalently, $h_1^*(\Delta_{i,j})=0$. 
Therefore, it follows from Proposition~\ref{prop:emplex} that all $h_2^*(\Delta_{i,j})$'s are equal for any $i,j$. Let $c \coloneqq h_2^*(\Delta_{i,j})$. 

Let $\ell=\min\{i:a_i>0\}$ for fixed $a_0,a_1,\ldots,a_{d-1}$. Let $P'$ be the lattice polytope
\[
P' \coloneqq \conv(\{\eb_i,\eb_i+a'_i\eb_d : 0 \leq i \leq d-1\}) \subset P,
\]
considered with respect to $N$, where $a'_i \coloneqq a_i$ for $i \neq \ell$ and $a'_\ell \coloneqq a_\ell-1$. Equivalently, $P'$ is the lattice polytope given as (the closure of) $P \setminus \Delta_{\ell,a_\ell}$. For simplicity we introduce the notation $\Delta \coloneqq \Delta_{\ell,a_\ell}$. Let $S$ be the $(d-1)$-dimensional simplex given as the intersection $P' \cap \Delta$. Then, from the inclusion-exclusion formula $\ehr_P(k)=\ehr_{P'}(k)+\ehr_{\Delta}(k)-\ehr_{S}(k)$, it follows that 
\[
\sum_{i=0}^d h_i^*t^i = \sum_{i=0}^d h^*_i(P')t^i+\sum_{i=0}^d h^*_i(\Delta)t^i-(1-t)\sum_{i=0}^{d-1} h^*_i(S)t^i.
\]
In particular we obtain the following: 
\begin{align*}
h_1^*=&h^*_1(P')+h^*_1(\Delta)-h^*_1(S)+1,\\
h_2^*=&h^*_2(P')+h^*_2(\Delta)-h^*_2(S)+h^*_1(S),\\
h_3^*=&h^*_3(P')+h^*_3(\Delta)-h^*_3(S)+h^*_2(S).
\end{align*}
By $h_3^*=0$ and by monotonicity, we obtain that $h^*_3(P')=h^*_3(\Delta)=h^*_3(S)=h^*_2(S)=0$. 
Moreover, we also know $h^*_1(\Delta)=h^*_1(S)=0$. In addition, $h^*_2(\Delta)=c$ by definition. Therefore,
\[
h_1^*=h^*_1(P')+1 \;\;\text{ and }\;\; h_2^*=h^*_2(\Delta)+c.
\]

Now, we iterate this construction. More precisely, one can replace $P$ by $P'$ and perform the same computation for this new starting polytope. We can do this until $P'$ becomes $\Delta_{d',1}$, where $d'=\max\{i:a_i>0\}$. 
This process stops after $b-1$ iterations, where $b=\sum_{i=0}^{d-1}a_i$. 
Hence, we eventually obtain the following:
\[
h_1^*=b-1 \;\; \text{ and }\;\; h_2^*=bc.
\]
This says that $h_2^*$ is divisible by $h_1^*+1$, as desired. 
\end{proof}

\section{Necessary condition for $\eqref{main:3}$}\label{sec:special}
Finally, we study necessary conditions for condition $\eqref{main:3}$ of Theorem~\ref{thm:main}.

\begin{prop}
\label{prop:last}
Let $P$ be a $d$-dimensional lattice polytope with $h^*$-polynomial $1+7t+t^2+h^*_4t^4 + \cdots $. 
Then its spanning polytope $\widetilde{P}$ is unimodular equivalent to the polytope obtained by the $(d-2)$-fold iterations of the lattice pyramid construction over the triangle $\conv(\{\orig,3\eb_1,3\eb_2\})$.
\end{prop}
\begin{proof}
By Theorem~\ref{thm:HKN}, the $h^*$-polynomial of $\widetilde{P}$ is of the form ${h}^*_{\widetilde{P}} = 1 + 7 t + \widetilde{h_2^*}t^2$. 
Suppose that $\widetilde{h_2^*}=0$, then since $\widetilde{P}$ cannot be an exceptional simplex, it must be a Lawrence prism. Then $h_2^*$ is divisible by $h_1^*+1$ by Proposition \ref{prop:lll}. In particular, $(h_1^*,h_2^*)=(7,1)$ never happens. 
Since $0<\widetilde{h_2^*} \leq h_2^*=1$, we assume that $\widetilde{h_2^*}=1$. By Theorem~\ref{thm:generalizedScott}, $\widetilde{P}$ is unimodularly equivalent to a polytope obtained via multiple lattice pyramids over $\conv(\{\orig, 3\eb_1',3\eb_2'\})$, where $\eb_1',\eb_2'$ are two elements of a basis for the affine sublattice $\widetilde{N}$. 
\end{proof} 

This completes the proof of Theorem~\ref{thm:main}. 

We conclude by noting that polytopes satisfying condition $\eqref{main:3}$ of Theorem~\ref{thm:main} are not necessarily obtained via a join.

\begin{ex}\label{ex:non_join}
Let $v_0=\orig$, $v_1=3\eb_1,v_2=3\eb_2$ and $v_i=\eb_i$ for $i=3,\ldots,8$, where $\eb_i$ is a basis for $\Z^8$, and let 
$\Delta=\conv(\{v_i:i=0,1,\ldots,8\})$ be a lattice simplex 
with respect to the lattice $N=\Z\cdot(1/2,1/2,1/2,1/2,1/2,1/2,1/2,1/2)+\Z^8$. Then 
\[
\Lambda_\Delta=\langle (1/3,2/3,0,0,\ldots,0), (1/3,0,2/3,0,\ldots,0) \rangle \oplus 
\langle (0,1/2,1/2,\ldots,1/2) \rangle,
\]
where $\langle \alpha,\ldots,\beta \rangle$ denotes the abelian group generated by $\alpha,\ldots,\beta$. 
Thus, $h_\Delta^*(t)=1+7t+t^2+6t^4+3t^5$ by \eqref{eq:h^*}. 
Notice that $\Delta$ cannot be a join of some two lower-dimensional lattice polytopes by Lemma \ref{lem:join}. 
\end{ex}


\bibliographystyle{plain}

\begin{thebibliography}{10}

\bibitem{BH10}
Victor Batyrev and Johannes Hofscheier.
\newblock A generalization of a theorem of {G}. {K}. {W}hite.
\newblock \href{https://arxiv.org/abs/1004.3411}{\texttt{arXiv:1004.3411
  [math.CO]}}, 2010.

\bibitem{BN07}
Victor Batyrev and Benjamin Nill.
\newblock Multiples of lattice polytopes without interior lattice points.
\newblock {\em Mosc. Math. J.}, 7(2):195--207, 349, 2007.

\bibitem{BR15}
Matthias Beck and Sinai Robins.
\newblock {\em Computing the continuous discretely}.
\newblock Undergraduate Texts in Mathematics. Springer, New York, second
  edition, 2015.
\newblock Integer-point enumeration in polyhedra, With illustrations by David
  Austin.

\bibitem{Ehr62}
Eug\`ene Ehrhart.
\newblock Sur les poly\`edres rationnels homoth\'etiques \`a {$n$}\ dimensions.
\newblock {\em C. R. Acad. Sci. Paris}, 254:616--618, 1962.

\bibitem{HT09}
Martin Henk and Makoto Tagami.
\newblock Lower bounds on the coefficients of {E}hrhart polynomials.
\newblock {\em European J. Combin.}, 30(1):70--83, 2009.

\bibitem{Hib94}
Takayuki Hibi.
\newblock A lower bound theorem for {E}hrhart polynomials of convex polytopes.
\newblock {\em Adv. Math.}, 105(2):162--165, 1994.

\bibitem{HKN16}
Johannes Hofscheier, Lukas Katth\"an, and Benjamin Nill.
\newblock Ehrhart theory of spanning lattice polytopes.
\newblock \href{https://arxiv.org/abs/1608.03166}{\texttt{arXiv:1608.03166
  [math.CO]}}, 2016.

\bibitem{Sco76}
Paul~R. Scott.
\newblock On convex lattice polygons.
\newblock {\em Bull. Austral. Math. Soc.}, 15(3):395--399, 1976.

\bibitem{Sta80}
Richard~P. Stanley.
\newblock Decompositions of rational convex polytopes.
\newblock {\em Ann. Discrete Math.}, 6:333--342, 1980.
\newblock Combinatorial mathematics, optimal designs and their applications
  (Proc. Sympos. Combin. Math. and Optimal Design, Colorado State Univ., Fort
  Collins, Colo., 1978).

\bibitem{Sta91}
Richard~P. Stanley.
\newblock On the {H}ilbert function of a graded {C}ohen-{M}acaulay domain.
\newblock {\em J. Pure Appl. Algebra}, 73(3):307--314, 1991.

\bibitem{Sta93}
Richard~P. Stanley.
\newblock A monotonicity property of {$h$}-vectors and {$h^*$}-vectors.
\newblock {\em European J. Combin.}, 14(3):251--258, 1993.

\bibitem{Sta09}
Alan Stapledon.
\newblock Inequalities and {E}hrhart {$\delta$}-vectors.
\newblock {\em Trans. Amer. Math. Soc.}, 361(10):5615--5626, 2009.

\bibitem{Sta16}
Alan Stapledon.
\newblock Additive number theory and inequalities in {E}hrhart theory.
\newblock {\em Int. Math. Res. Not. IMRN}, (5):1497--1540, 2016.

\bibitem{Tre10}
Jaron Treutlein.
\newblock Lattice polytopes of degree 2.
\newblock {\em J. Combin. Theory Ser. A}, 117(3):354--360, 2010.

\end{thebibliography}

\end{document}